\documentclass{article}
\usepackage[utf8]{inputenc}

\usepackage{amsmath,amssymb,amsthm,bbm}
\usepackage{mathtools}
\usepackage{mathabx}
\usepackage{cite}
\usepackage{amsrefs}
\usepackage{color}
\usepackage{tikz}

\usetikzlibrary{angles,quotes}

\theoremstyle{definition} \newtheorem{dfn}{Definition}[section]
\theoremstyle{definition} 
\theoremstyle{definition} \newtheorem*{axi*}{Axiom}
\theoremstyle{definition} \newtheorem{prp}[dfn]{Proposition}
\theoremstyle{definition} \newtheorem{lem}[dfn]{Lemma}
\theoremstyle{definition} \newtheorem{thm}[dfn]{Theorem}
\theoremstyle{definition} \newtheorem{cor}[dfn]{Corollary}
\theoremstyle{definition} 
\theoremstyle{definition} \newtheorem{que}[dfn]{Question}
\theoremstyle{remark} 

\theoremstyle{remark}  

\definecolor{darkgreen}{rgb}{0,.6,0}

\newcommand \dd {\mid\mid}

\newcommand \N {\mathbb N}
\newcommand \Z {\mathbb Z}
\newcommand \Q {\mathbb Q}
\newcommand \R {\mathbb R}
\newcommand \C {\mathbb C}

\newcommand \D {\mathcal D}

\newcommand \lcm {\text{lcm}}

\newcommand \la {\lambda}
\newcommand \La {\Lambda}

\if10     
\usepackage[mathlines]{lineno}
\newcommand*\patchAmsMathEnvironmentForLineno[1]{%
	\expandafter\let\csname old#1\expandafter\endcsname\csname #1\endcsname
	\expandafter\let\csname oldend#1\expandafter\endcsname\csname end#1\endcsname
	\renewenvironment{#1}%
	{\linenomath\csname old#1\endcsname}%
	{\csname oldend#1\endcsname\endlinenomath}}%
\newcommand*\patchBothAmsMathEnvironmentsForLineno[1]{%
	\patchAmsMathEnvironmentForLineno{#1}%
	\patchAmsMathEnvironmentForLineno{#1*}}%
\AtBeginDocument{%
	\patchBothAmsMathEnvironmentsForLineno{equation}%
	\patchBothAmsMathEnvironmentsForLineno{align}%
	\patchBothAmsMathEnvironmentsForLineno{flalign}%
	\patchBothAmsMathEnvironmentsForLineno{alignat}%
	\patchBothAmsMathEnvironmentsForLineno{gather}%
	\patchBothAmsMathEnvironmentsForLineno{multline}%
}
\linenumbers
\fi

\newcommand\blfootnote[1]{%
	\begingroup
	\renewcommand\thefootnote{}\footnote{#1}%
	\addtocounter{footnote}{-1}%
	\endgroup
}

\title{Spectral sets and tiles in \(\Z_p^2 \times \Z_q^2\)}
\author{Thomas Fallon \thanks{The Graduate Center, the City University of New York \newline $\textrm{ \ \ \ \ }$ e-mail: tfallon@gradcenter.cuny.edu}, Gergely Kiss\thanks{Alfr\'ed R\'enyi Institute of Mathematics, Analysis Department \newline $\textrm{ \ \ \ \ }$ e-mail: kigergo57@gmail.com }, G\'abor Somlai\thanks{E\"otv\"os Loránd University, Department of Algebra and Number Theory \newline $\textrm{ \ \ \ \ }$ e-mail: zsomlei@caesar.elte.hu}}
\date{\today}

\begin{document}
	
	\maketitle
	
	\begin{abstract} In this paper we study Fuglede's conjecture on the direct product of two cyclic groups. After collecting the known results we prove that Fuglede's conjecture holds on $\Z_{pq} \times \Z_{pq} \cong \Z_p^2 \times \Z_q^2$. 
		
	\end{abstract}
	
	\blfootnote{%
		} 	
	\section{Introduction}
	The study of exponential bases and the Fuglede conjecture attracted the attention of researchers decades ago and still now it is an active area of research. In 1974 Bent Fuglede was studying Segal's problem on partial differential operators \cite{F1974} and this led him to find a connection between spectral sets and tiles. His conjecture was stated originally on $\R^n$ as follows. Let $\Omega \subset \R^n$ be bounded measurable set with positive Lebesgue measure. $\Omega$ is called {\it spectral} if there exists a set $\Lambda\subset \R^n$ such that $\{e^{\langle x,\lambda \rangle}:\lambda \in \Lambda\}$ is an orthogonal basis of $L^2(\Omega)$. In this case $\La$ is called a {\it spectrum} of $\Omega$. $\Omega$ is called a {\it tile} if there exists a set $T\subset \R^n$ such that the translates $\Omega+t \ (t\in T)$ cover $\R^n$ exactly once up to a set of measure zero. In this case $T$ is called the tiling complement of $\Omega$. Fuglede's conjecture states that $\Omega$ is spectral if and only if $\Omega$ is a tile. 
	
	After some valuable positive results on $\R^n$ (e.g. \cites{F1974, IKP99, L2001}), Tao \cite{T2004} disproved the conjecture in 2004 by exhibiting a spectral set which is not a tile in $\mathbb{R}^n$ for $n\ge 5$. This result spurred intense research in this area. Namely, some non-tiling spectral sets were found in $\R^n$ for $n\ge 4$ \cite{Matolcsi2005} and later for $n\ge 3$ \cite{KMkishalmaz}, and  non-spectral tiles were shown to exist in $\mathbb{R}^n$ for $n\ge 3$ \cite{FMM2006} (for further references see \cites{FR2006, KM2}).
	Remarkably, both directions of the conjecture are still open in $\R$ and $\R^2$. 
	All of the aforementioned counterexamples to the conjecture in $\R^n$ are based on {\bf Fuglede's conjecture on finite abelian groups}.
	
	Let $G$ be a finite abelian group and $\widehat{G}$ the dual group (i.e. the set of irreducible representations of $G$, which can be considered as a group) and it is isomorphic to $G$. The elements of $\widehat{G}$ can be indexed by the elements of $G$. Then $S\subset G$ is {\it spectral} if and only if there exists a $\La\in G$ such that ($\chi_l)_{l\in \La}$ is an orthogonal base of complex valued functions defined on $S$.
	For a finite group $G$ and a set $S\subseteq G$ we say that $S$ is {\it a tile} of $G$ if there is a $T \subset G$ such that $S+T=G$ and $|S|\cdot |T|=|G|$. This we denote by $S\bigoplus T$. The discrete version of the problem can be formulated as follows. 
	\begin{que}
		For a given abelian group $G$ is that true that the spectral sets and the tiles coincide?
	\end{que}

	Note that if $G$ is a finite abelian group, then $\La$ is a spectrum for $\Omega$ if and only if $\Omega$ is a spectrum for $\La$ (under the identification $\widehat{\widehat{G \ }}=G$), and so in this case we say that $(\Omega,\La)$ is a {\it spectral pair} and we have $|\Omega|=|\La|$.
	Although the conjecture is known to be false in general, many positive examples are also known (see below) which indicates the following problem.
	
	The connections between spectral sets and tiles on $\mathbb{Z}$ and on finite cyclic groups $\mathbb{Z}_N$ were settled in the last two decades by Coven and Meyerowich \cite{CM1999}, \L aba \cite{L2002} and Dutkay and Lai \cite{DutkayLai}. It was shown in \cite{DutkayLai} that the tile-spectral direction of Fuglede's conjecture on $\R$ holds if and only if it holds for $\Z$, which happens if and only if it holds for every finite cyclic group.
	They also showed that if the spectral-tile direction holds on $\R$, then it holds for $\Z$ and for every cyclic group. On the other hand, the reverse implications are still open.
	
	Several attempts have been made to prove that the discrete version of Fuglede's conjecture holds for some special cyclic groups. It was shown that if $G=\Z_N$ and $N$ is square-free or $N=p^mq^n$ or $N=p^n d$, where $d$ is square-free, then any tile of $G$ is spectral \cites{L2002, MK20, Shi18}.
	For the spectral-tile direction only partial results are known. If $G=\Z_N$ and one of the following conditions holds: $N=p^n$ \cite{L2002}, $N=p^nq$ \cite{MK17}, $N=p^nq^2$ \cite{KMSV2020}, $N=pqr$ \cite{Shi18}, $N=p^2qr$ \cite{So2019} or (quite recently) $N=pqrs$ \cite{KMSV2021} or $N=p^nq^k$, where $\min{(n,k)}\le 6$ \cite{MK20}, then a subset of $G$ is spectral if and only if it is a tile.

	{\bf Fuglede conjecture on the direct product of abelian groups.}
	Tao's \cite{T2004} example of a spectral set which does not tile comes from a non-tiling spectral set in $\Z_3^5$. On the other hand, it was proved in \cite{Z_p^2} that  Fuglede's conjecture holds in $\Z_p^2$ for every prime $p$. 
	Recently, the investigation of the problem on elementary abelian $p$-groups has become of interest. For odd primes $p$, non-spectral tiles have been exhibited in $\Z_p^5$ \cite{atenetal} and in $\Z_p^4$ by Ferguson and Sothanaphan \cite{fergsoth2019} and independently by Mattheus \cite{SM}.
	If $p=2$, then the situation is slightly different. It was shown that Fuglede's conjecture fails for $\Z_2^d$ if $d\ge 10$ \cite{fergsoth2019}, and holds if $d\le 6$ \cites{fergsoth2019,fsadd}. For $7\le d\le 9$ the answer is not known.
	Fuglede's conjecture for $\Z_p^3$ when $p$ is an odd prime is still wide open, although partial results have been obtained recently for $p \le 5$ in \cite{birklbauer} and for $p \le 7$ in \cite{zp37}.
	Note that the tile-spectral direction of the conjecture holds for $\Z_p^3$, see \cite{atenetal}.
	
	Clearly, Fuglede's conjecture does not hold on $\R^2$ if Fuglede's conjecture does not hold on an abelian group of the form $\Z_m\times\Z_n$. Therefore, it is reasonable to investigate the problem for such direct products.
	R. Shi verified the conjecture for $\Z_{p^2} \times \Z_p$ \cite{Shi2020}.
	It was proved by the second and the third author that Fuglede's conjecture holds on $\mathbb{Z}_p^2 \times \mathbb{Z}_q\cong \Z_{pq} \times \Z_p$ \cite{KS2021}.

	These previously presented results on some direct product of cyclic groups and the purpose of finding a potential counterexample in $\R^2$ motivate our main result which is the following. 
	\begin{thm}\label{t1}
		Fuglede's conjecture holds for $\mathbb{Z}_p^2 \times \mathbb{Z}_q^2\cong \Z_{pq} \times \Z_{pq}$, where $q$ and $p$ are different primes.
	\end{thm}
	Moreover as we will see in Section \ref{sec5} that if a set $A$ tiles $\mathbb{Z}_p^2 \times \mathbb{Z}_q^2$, then it tiles with a subgroup as well. 
	\section{Preliminaries and Notation}
	\begin{dfn}
		Given a finite abelian group \(G\), we can define the {\it{character functions}} as homomorphisms \(\chi \colon G \to S^1 \leq \C^*\).
	\end{dfn}
	
	\begin{dfn}
		Given an finite abelian group \(G\), its {\it dual group} \(\widehat{G}\) is the collection of its {\it{character functions}} whose operation is point-wise multiplication.
	\end{dfn}
	
	\begin{dfn}
		The {\it order} of a character \(\chi\) is just as in any group the smallest \(n\) such that \(\chi^n\) is the trivial character.
	\end{dfn}
	
	It is known that after choice of generators, \(G\) can be viewed as \(G = \bigoplus_{i=1}^N \Z_{n_i}^{d_i}\).
	The exponent of this group is then \(M = \lcm_{i=1}^N n_i\).
	There is a dot product $\langle x, y \rangle\colon G^2 \to \Z_M$, where \(\langle x, y \rangle \coloneqq \sum_{i=1}^N \frac{M}{n_i} (x_i \cdot y_i)\).
	This product can then be used to give an isomorphism between \(G\) and \(\widehat{G}\) where for \(g \in G\), \(\chi_g \coloneqq \left(x \mapsto e^{\frac{i\pi}{M} \langle x, g \rangle}\right)\).

	Using the character functions defined above, we can define the Fourier transform of a function \(f \in L^2(G)\) as
	\[
	\hat{f}(\chi) \coloneqq |G|^{-1} \sum_{g \in G} f(g) \overline{\chi(g)} \quad \quad (\chi \in \widehat{G})
	\]
	with inverse Fourier transform of a function \(h \in L^2(\widehat{G})\)
	\[
	\check{h}(g) \coloneqq \sum_{\chi \in \widehat{G}} h(\chi)\chi(g) \quad  (g \in G).
	\]
	
	Note that for an indicator function $1_S$ of \(S \subseteq G\), its Fourier transform is \(\hat{1}_S(\chi) = |G|^{-1} \sum_{g \in S} \overline{\chi(g)}\).  
	
	A multiset on a group $G$ is a function from $G$ to $\mathbb{N}$.
	\begin{dfn}
		Related to this we can have a character function act on a set in a similar way and define
		\[
		\chi(S) \coloneqq \sum_{s \in S} \chi(s).
		\]
	\end{dfn}
	It should be noted that \(\chi(S) = |G| \hat{1}_S(\chi)\), and \(\chi(S) = 0\) exactly when \(\hat{1}_S(\chi) = 0\).
	
	\begin{dfn}
		A convenient notation to use when \(S \subseteq G\) is to say \(n \dd |S|\) when \(\gcd(|S|, |G|) = n\). 
	\end{dfn}
	
	\begin{dfn}
		Where \(S \subseteq G\) and \(\alpha \in \widehat{G}\), we denote the projection of \(S\) as a multi-set to \(G / \alpha^\perp\) by \(S_\alpha\), where $\alpha^\perp$ is the kernel of $\alpha$. More precisely let 
		\[ S_\alpha(i)=\sum_{ \{ x ~\colon~ \alpha(x)=i \} } S(x).\]
		Note that $S_{\alpha}$ depends only on the kernel of $\alpha$.
		This will often be used because the value of \(\hat{1}_S(\alpha)\) depends only on \(S_\alpha\). 
	\end{dfn}
	
	\begin{dfn}
		Given an element or a subset \(S\) in \(G\), we talk about the perpendicular group \(S^\perp\) in \(\widehat{G}\).  Define the perpendicular group to \(S\) as 
		\[
		S^\perp \coloneqq \{\chi \in \widehat{G} \colon \chi(s) = 1 ~ (\forall s \in S)\}
		\]
	\end{dfn}

	\begin{dfn}
		Where \(S \subset G\), we denoted the projection of \(S\) as a multiset to the Sylow \(p-\)group by \(S_p\).  Note that since \(G\) is abelian, this is well defined. \end{dfn}
	
	\begin{dfn} \label{defdir} Let $G$ be a finite abelian group. For $v, w\in G$, we write $v\sim w$ if $v$ and $w$ generate the same subgroup of $G$. Clearly, $\sim$ is an equivalence relation. Its equivalence classes are called the directions in $G$ and they are represented by $[v]$ for every $v\in G$.
		Let $S\subseteq G$ be a set and $v\in G$. We say that $S$ determines the direction $[v]$, if there is a $w\in S-S\subseteq G$ such that $w\sim v$. The set of directions determined by $S$ is denoted by $\mathcal{D}(S)$. For $G=\Z_p^2\times \Z_q^2$ we denote an element by $(a,b)$ and the corresponding direction by $[a,b]$, where $a\in \Z_p^2$ and $b\in \Z_q^2$. For any element $x=(a,b)\in \Z_p^2\times \Z_q^2$, where $a\in\Z_p^2$ and $b\in \Z_q^2$, we write $x=a+b$. 
	\end{dfn}
	
	\section{Useful lemmas}
	In this section we collect a few facts that will be used all throughout the paper. 
	
	We recall that for finite abelian groups $G$ if $S\subseteq G$ is a spectral set and $\La\subseteq \widehat{G}$ is its spectrum, then $S\subseteq G=\widehat{\widehat{G \ }}$ is a spectrum for $\Lambda$, so $\Lambda$ is a spectral set in $\widehat{G}$. In this case $(S,\La)$ is spectral pair as well as $(\La, S)$, and $|S|=|\La|$. For a spectral pair $(S,\La)$, \(\hat{1}_S(\chi)=0\) for all \(\chi \in \Lambda - \Lambda\), by changing the role of $S$ and $\La$, we also have and \(\hat{1}_{\La}(\chi)=0\) for all \(\chi \in S - S\). 

	\begin{lem}\label{lemeqdsubgrp}
		Let \(A\) be a multiset on a finite abelian group \(G\) and \(H \leq \widehat{G}\), then the following are equivalent
		\begin{itemize}
			\item \(A\) is equidistributed among cosets of \(H^\perp\), that is \(\sum_{g \in H^\perp + x} A(g)\) is constant.
			\item \(\hat{A}(\chi) = 0\) for all \(\chi \neq 0 \in H\).
		\end{itemize}
	\end{lem}
	\begin{proof}
		The multiset \(A\) is equidistributed among cosets of \(H^\perp\) when the convolution \(A \ast 1_{H^\perp}\) is constant.  The inverse Fourier transform of the indicator of \(H\) is a multiple of the indicator of \(H^\perp\), so taking the Fourier transform of \(A \ast \hat{1}_{H^\perp}\) gives \(\hat{A} \cdot 1_{H}\) is a multiple of \(\delta_0\).
	\end{proof}
	
	\begin{cor}\label{corequigen}
		If \(S\) is a spectral set in \(G\) and \(S\) determines every direction in a subgroup \(H \leq G\), then \(|H| \mid |S|\).
	\end{cor}
	\begin{proof}
		If \(S\) is spectral, it must have a spectrum \(\Lambda\) where \(|\Lambda| = |S|\) and for any nonzero \(g \in S - S\), \(\hat{1}_\Lambda(g) = 0\).  Since \(S\) determines every direction of \(H\), \(\hat{1}_\Lambda\) vanishes on all nonzero elements of \(H\).
		From Lemma \ref{lemeqdsubgrp}, the means \(\La\) must be equidistriubuted among cosets of \(H^\perp\), of which there are \(|H|\) many of them, and so \(|H| \mid |\La| = |S|\).
	\end{proof}
	The previous statement will mostly be used in the paper in the following special form. 
	\begin{lem}\label{lemp2mid}
		Assume $A$ is a multiset on $\Z_r^2$, where $r$ is a prime with $\chi(A)=0$ for every $\chi \in \widehat{\Z_r^2}\setminus \{ \hat{0} \} $. Then $r^2 \mid |A|$. 
	\end{lem}
	The property introduced in the next corollary is usually called {\it equidistributivity property} of a spectral set $S$ in direction $\alpha$. 
	\begin{cor}\label{cor:equi}
		Let $S$ be a spectral set. If $\alpha\in \La-\La$ for some $\alpha\in \widehat{\Z_r}\setminus \{0 \}$, where $r$ is a prime, then $r\mid|S|=|\La|$. Moreover it follows from the previous argument that $S$ is equidistributed among the cosets of $\alpha^{\perp}$, see also \cite{KS2021}. 
	\end{cor}
	The following Lemma is folklore. One can find proper proof of the statement in \cite{La} and a generalization of this result is found in \cite{KMSV2020}. 
	\begin{lem}\label{lemcuberule1}
		Let $A$ be a multiset on $\Z_{pq}$. Assume $\chi(A)=0$ for some character of order $pq$. Then $A$ is the nonnegative sum of $\Z_p$ cosets and $\Z_q$ cosets. Moreover the converse is also true. 
	\end{lem}
	
	\begin{lem}\label{lemcharacterequiv}
		Let $G$ be a finite abelian group. Let us assume $\chi_v(S)=0$ for some irreducible representation $\chi_v$ of $G$. Then for $\chi_{vk}(S)=0$ if $gcd(k,|G|)=1$. 
	\end{lem}
	\begin{proof}
		Applying a suitable element of the Galois group $Gal(\Q(\xi_{|G|}) \mid \Q)$ to the equation $\chi_v(S)=0$, where $\xi_{|G|}$ is a $|G|$'th root of unity, gives the result.
	\end{proof}
	The condition $gcd(k,pq)=1$ in Lemma \ref{lemcharacterequiv} is equivalent to $v\sim vk$ and this motivates the introduction of directions in Definition \ref{defdir}.	
	\section{Spectral Implies Tiling in \(\Z_p^2 \times \Z_q^2\)}
	
	For this direction, we split into cases based on divisibility of the spectral set \(S\).
	
	\medskip
	\noindent
	\underline{{\bf Case 1:} \(p^2 q \dd |S|\) (or \(p q^2 \dd |S|\))} \\
	From a pigeon hole argument, such an \(S\) either tiles by a \(\Z_q\) coset and has size \(p^2q\), or \(S\) determines every direction of order \(q\).  But if \(S\) determines every direction of order \(q\), then by Lemma \ref{lemp2mid}, \(q^2 \mid |S|\), contradicting our original assumption that \(p^2 q \dd |S|\).
	
	\medskip
	\noindent
	\underline{{\bf Case 2:} \(p^2 \dd |S|\) (or \(q^2 \dd |S|\))} \\
	Similar to the above case, a pigeon hole argument shows that either \(S\) tiles by \(\Z_q^2\) and has size \(p^2\), otherwise \(S\) determines some direction of order \(q\), which implies that \(q \mid |S|\) by Corollary \ref{cor:equi}, contradicting that \(p^2 \dd |S|\).	
	
	\medskip
	\noindent
	\underline{{\bf Case 3:} \(1 \dd |S|\)} \\
	In the case that \(|S| = 1\), then \(S\) tiles trivially by the whole space.  Otherwise \(|S|=|\La| > 1\), and for any nonzero \(\chi \in \Lambda - \Lambda\), \(\chi(S) = 0\).  If the order of \(\chi\) is \(p\) or \(q\), then by Corollary \ref{cor:equi}, \(|S|\) would have to be a multiple of \(p\) or \(q\) respectively, contradiction our assumption that \(1 \dd |S|\). So for all such characters, the order must be \(pq\).
	For such a \(\chi\), when projecting \(S\) to \( \left( \Z_p^2 \times \Z_q^2 \right)/ \chi^\perp \cong \Z_p \times \Z_q\), \(S\) must project to a sum of \(\Z_p\) and \(\Z_q\) cosets by Lemma \ref{lemcuberule1}.  
	If this projection were only \(\Z_p\) or \(\Z_q\) cosets, then \(|S|\) would be a multiple of \(p\) or \(q\), so there must be at least one of each in this projection.  Without loss of generality, assume \(p < q\).  
	The preimage of any of those \(\Z_q\) cosets is in a coset in \(G\) of a subgroup isomorphic to \(\Z_p \times \Z_q^2\) that contains \(q\) points of \(S\). Since this contains only \(p\) cosets of \(\Z_q^2\), there must be some \(\Z_q^2\) coset with 2 points from \(S\).
	Thus \(S\) determines a direction of order \(q\) and is spectral, hence \(|S|\) must be a multiple of \(q\), and once again contradicts our original assumption.

	\medskip
	\noindent
	\underline{{\bf Case 4:} \(p \dd |S|\) (or $q \dd |S|$)} \\
	Let us first assume $|S|=p$. It is clear that $S$ vanishes only on characters of order $p$ or $pq$, otherwise $q\mid |S|$ by Corollary \ref{cor:equi}, which is a contradiction. 
	Let $\chi$ be a character such that $\chi(S)=0$. 
	If $o(\chi)=p$, then $\langle \chi \rangle$ is a spectrum for $S$. If $\chi$ is of order $pq$, then the projection of $S$ to $(\Z_p^2 \times \Z_q^2)/\chi^{\perp}$ is the sum of $\Z_p$ cosets and  $\Z_q$ cosets by Lemma \ref{lemcuberule1}. Since $|S|=p$, this projection is just a $\Z_p$ coset. Hence $\chi^{q}(S)=0$ and 
	$\langle \chi^{q} \rangle$ is a spectrum for $S$.
	
	From now on we assume $|S|>p$.
	The spectrum \(\Lambda\) must project to \(\Z_p^2\) as a set so as not to determine any direction of order \(q\).
	If \(|S| = |\Lambda| \neq p\), then the projection of $\La$ determines every direction of order \(p\), so for every $0 \ne u \in \Z_p^2$ there is $v_u \in \Z_q^2$ 
	such that \(\chi_{u+v_u}(S) = 0\) by Lemma \ref{lemcharacterequiv}.
	
	We argue that $v_u \ne 0$ for some $0\ne u \in \Z_p^2$ since $p^2 \nmid |S|$. In this case 
	there is some character \(\chi_{u+v_u}\) of order \(pq\) where \(\chi_{u+v_u}(S) = 0\). Then \(S\) projects by \(\chi_{u+v_u}^\perp\) to a sum of \(\Z_p\) and \(\Z_q\) cosets by Lemma \ref{lemcuberule1}. Clearly, $S/\chi_{u+v_u}^\perp$ cannot only be the sum of $\Z_p$ cosets since that would imply $\chi_{u}(S)=0$, which is excluded. Similarly, $S/\chi_{u+v_u}^\perp$ is not the sum of $\Z_q$ cosets only since $q \nmid |S|$. 
	
	Since \(|S|\) is a multiple of \(p\), so must be the number of \(\Z_q\) cosets that appear. Hence we obtain $|S|>pq$.
	
	\begin{lem}\label{lem:obs}
		If \(A \subseteq \Z_p^2 \times \Z_q^2\) with 
		\(|A|\geq pq\), then either \(A\) is a tile and $|A|=pq$, or for all \(0\ne a \in \Z_p^2, 0\ne b\in \Z_q^2\) either \([a,0] \in \D(A)\) or \([0,b] \in \D(A)\) or \([a,b] \in \D(A)\). 
	\end{lem}
	\begin{proof}
		The subgroup $\langle (a,b) \rangle \cong \Z_p \times \Z_q$ is of index $pq$. 
		Either \(A\) has one point in each $\langle (a,b) \rangle$ coset and $A$ tiles, or there are two elements of \(A\) in one of those cosets.
		In that case, \(A\) must determine some direction of $\langle (a,b) \rangle$, which is either \([a,0] \in \D(A)\) or \([0,b] \in \D(A)\) or \([a,b] \in \D(A)\).
	\end{proof}
	
	By applying Lemma \ref{lem:obs} for $\La$ ($|\La|>pq$) we conclude that for every $0 \ne v \in \Z_q^2$ one of the followings holds: $\chi_{u+v}(S)=0$, $\chi_{u}(S)=0$ or $\chi_{v}(S)=0$. As 
	$\chi_{u}(S)=0$ and $\chi_{v}(S) = 0$ for every $0 \ne v \in \Z_p^2$ are excluded by our assumptions we obtain  $\chi_{u+v}(S)=0$ for every $0 \ne v \in \Z_p^2$.
	
	\begin{prp}\label{prop1}
		Let $T$ be a multiset $\Z_p \times \Z_q^2$ with $\chi_{x+y}(T)=0$ for every $0 \ne x \in \Z_p$ and $0 \ne y \in \Z_q^2$.
		For every $i \in \Z_p$ let $g_i\colon \Z_q^2\to\N $ be a function such that $g_i(z)=T(i+z)$ for all $z\in \Z_q^2$. Then $g_i-g_j$ is a constant function for every $i,j \in \Z_p$. 
	\end{prp}
	\begin{proof}
		We prove that $supp(\widehat{g_i-g_j})=supp(\hat{g_i}-\hat{g_j})=\{0\}$. Let $0 \ne v \in \Z_q^2$.   
		In order to calculate $\hat{g_i}(v)$ we first project $T$ to $\langle u,v\rangle$, where $0 \ne u \in \Z_p$. It follows from $\chi_{u+v}(T)=0$ that $T_{\langle u,v \rangle}$ is the sum of $\Z_p$ cosets and $\Z_q$ cosets. For every $i \in \Z_p$, $w \in \Z_q$ let $f_{i,v}(w)=T_{\langle u,v \rangle}(i+w)$. Since $T_{\langle u,v \rangle}$ is the sum of $\Z_p$ cosets and $\Z_q$ cosets, $f_{i,v}-f_{j,v}$ is constant for $i,j \in \Z_p$. 
		
		It is straightforward to verify that $\hat{g_i}(v)=\hat{f}_{i,v}(w)$ for some $w \ne 0$ and then $\hat{g_i}(v)-\hat{g_i}(v)=\hat{f}_{i,v}(w)-\hat{f}_{j,v}(w)=0$ by Lemma \ref{lemeqdsubgrp}, since $f_{i,v}-f_{j,v}$ is constant. 
	\end{proof}
	It is clear from the previous argument that for the projection $T$ of $S$ to $\langle u, \Z_q^2 \rangle$ the conditions of Proposition \ref{prop1} hold. Thus we obtain that for $g_i(z)=T(i+z)$ ($z\in \Z_q^2$), $g_i-g_j$ is a constant function for every $i,j \in \Z_p$. If $g_i=g_j$ for every $i,j \in \Z_p$, then $\chi_u(S)=0$, a contradiction.
	
	Let $g_0$ denote the function which is minimal among $g_i$'s. If $g_0(z)=0$ for all $z\in \Z_q^2$, then $q^2 \mid |S|$, which contradicts the fact that $p \mid\mid |S|$. If $g_0(z)>0$ for some $z\in \Z_q^2$, then $g_i(z)>0$ for every $i \in \Z_p$ and we have seen that $g_j(z)>1$ for some $j \ne 0$. Thus $z+\Z_p^2$ contains at least $p+1$ elements of $S$, implying that $|S|$ determines every direction of $\Z_p^2$ and hence $p^2\mid |\La|=|S|$ by Corollary \ref{cor:equi}, a contradiction.

	\subsection{Case 5: $pq \dd |S|$}
	The case when $pq \dd |S|$ is treated separately simply because the argument needed here is more complicated. 

	\begin{prp}
		If \(S \subseteq \Z_p^2 \times \Z_q^2\) is spectral and \(|S| = pq\), then \(S\) is a tiling set.
	\end{prp}
	\begin{proof}
		
		Since $p^2 \nmid |S|=|\La|$, there is an $a$ of order $p$ with $\hat{1}_{\La}(a,0) \ne 0$.
		
		By Lemma \ref{lem:obs}, either \(S\) is a tile or \(S\) determines at least one of \([a,0]\), \([0,b]\) and \([a,b]\) for every \(0 \ne b \in \Z_q^2\). Note that $\hat{1}_{\La}(a,0) \ne 0$ implies \([a,0]\not\in \mathcal{D}(S)\).
		Thus \([0,b]\in \mathcal{D}(S)\) or \([a,b]\in \mathcal{D}(S)\). 
		
		If \([a,b] \in \D(S)\), then for the spectrum \(\Lambda\) of \(S\), \(\hat{1}_\Lambda(a,b) = 0\).
		By Lemma \ref{lemcuberule1}, $\La$ is the sum of $\Z_p$ cosets and $\Z_q$ cosets so $|\La|=kp+lq$ for some $k,l \in \N$.
		Since \(|\Lambda| = pq\) then either $k=q$ and $l=0$ or $k=0$ and $l=p$, implying  \(\hat{1}_\Lambda(a,0) = 0\) or \(\hat{1}_\Lambda(0,b) = 0\).  
		
		As \(\hat{1}_\Lambda(a,0)\neq 0\) we have \(\hat{1}_\Lambda(0,b)=0\) for all \(0\ne b\in \Z_q^2\), implying that \(q^2 \mid |\Lambda|\), a contradiction.
		So \(\hat{1}_\Lambda\) must vanish on \((a,0)\), but \(a\) was arbitrary, implying \(\hat{1}_\Lambda\) vanishes on every element of order \(p\), another contradiction.
		Therefore \(S\) must tile by a \(\Z_p \times \Z_q\) subgroup.
	\end{proof}
	\begin{lem}
		If \(pq \dd |S|\) and \(|S| \ge pq \min\{p,q\}\), then $S$ is not spectral.
	\end{lem}
	\begin{proof}
		Without loss of generality assume \(p<q\).  For a contradiction, we assume \(p^2 q < |S|\).  For any \(\Z_q\) subgroup of \(G\), there are \(p^2 q\) cosets so \(S\) must have at least two points in one of them.  As this \(\Z_q\) subgroup was chosen arbitrarily, \(S\) determines every direction of order \(q\).  Since \(S\) is a spectrum for $\Lambda$, this means that \(q^2 \mid |\La|=|S|\) giving a contradiction.
	\end{proof}
	From now on we assume $pq \dd |S|$ and $pq<|S|<pq \min\{p,q\}$.
	\begin{lem}\label{lemp2q2}
		There are $0 \ne u \in \Z_p^2$ and $0 \ne v \in \Z_q^2$ such that $\lambda u \not\in S-S$ and $\mu u \not\in S-S$ for any $\lambda \in \Z_p^*=\Z_p\setminus \{ 0\}$ and $\mu \in \Z_q^*=\Z_q\setminus \{ 0\}$.
	\end{lem}
	\begin{proof}
		Otherwise $p^2 \mid |S|$ or $q^2 \mid |S|$ by Lemma \ref{lemp2mid}.
	\end{proof}
	\begin{lem}\label{lema}
		For every $0 \ne u \in \Z_p^2$ and $0 \ne v \in \Z_q^2$ we have $\chi_{u+x_u}(S)=0$ and $\chi_{v+y_v}(S)=0$ for some $x_u \in \Z_q^2$ and $y_v \in \Z_p^2$. 
	\end{lem}
	\begin{proof}
		We will only prove that $\chi_{u+x_u}(S)=0$ for some $x_u \in \Z_q^2$. The other case is similar. 
		
		Let $S_p$ denote the multiset, which is the projection of $S$ to $\Z_p^2$. It follows from Lemma \ref{lemp2q2} that none of the $\Z_q^2$ cosets contains more than $q$ elements of $S$. Thus, as $|S_p|=|S|>pq$ we have $|supp(S_p)|>p+1$ and then $supp(S_p)$ determines every direction in $\Z_p^2$. 
	\end{proof}
	
	\begin{prp}\label{lemleaf}
		Assume for every $0 \ne u \in \Z_p^2$ we have $\chi_{u+x_u}(S)=0$  for some $x_u \in \Z_q^2$.
		Then $S_p=c \Z_p^2 + q D$, where $c$ is a nonnegative integer and $D$ is a multiset on $\Z_p^2$. 
	\end{prp}
	\begin{proof}
		
		For any $0\ne u\in\Z_p^2$ and $x\in\Z_p^2$ the cosets $x+\langle u \rangle \in \Z_p^2$ are called lines. Two lines are parallel if their intersection is empty. 
		
		We intend to prove that the multiset $S_p$ defined on $\Z_p^2$ is constant$\pmod{q}$. 
		First we show that if $L_1$ and $L_2$ are parallel lines in $\Z_p^2$, then $$\sum_{x \in L_1} S_p(x) \equiv \sum_{x \in L_2} S_p(x) \pmod{q}.$$ 
		Indeed, write $L_1=x_1+\langle u' \rangle$ and $L_2=x_2+\langle u' \rangle$. Let $0\ne u \in \Z_p^2$ be orthogonal\footnote{i.e., the dot product introduced on p.3 of $u$ and $u'$ is 0.} to $u'$.    
		If $\chi_u(S)=0$, then $\sum_{x \in L_1} S_p=\sum_{x \in L_2} S_p$. On the other hand, if $\chi_{u+x_u}(S)=0$, where $0 \ne x_u \in \Z_q^2$, them we first project $S$ to $\langle u,x_u\rangle$. The projection is the sum of $\Z_p$ cosets and $\Z_p$ cosets by Lemma \ref{lemcuberule1} so we obtain $\sum_{x \in L_1} S_p \equiv \sum_{x \in L_2} S_p \pmod{q}$. 
		The statement above follows from our conditions that at least one of the two cases holds. 
		
		Since $p$ and $q$ are different primes we obtain $$\sum_{x \in L} S_p(x) \equiv \frac{1}{p} \sum_{x \in \Z_p^2} S_p(x) \pmod{q}$$ for every line $L$ in $\Z_p^2$. 
		
		Let $x \in \Z_p^2$ and let $\mathcal{L}_x$ be the set of lines containing $x$. It follows from the previous observation that \[\sum_{L \in \mathcal{L}_a} \sum_{y\in  L } S_p(y)\equiv \sum_{L \in \mathcal{L}_b} \sum_{y\in  L } S_p(y) ~ \pmod{q} \]
		for every $a, b \in \Z_p^2$. This gives 
		\[ pS_p(a)+\sum_{x \in \Z_p^2}S_p(x) \equiv pS_p(b)+\sum_{x \in \Z_p^2}S_p(x) \pmod{q}.\]
		Since $(p,q)=1$ we have $S_p$ is constant$\pmod{q}$. Thus $S_p=c \Z_p^2+qD$ for some $c \in \Z$ and multiset $D$. We could have more than one choice for $c$ and $D$, but if $c=\min\{ S(x) \mid x \in \Z_p^2\}$, then $c \ge 0$ and $D$ is a multiset and this choice of $c$ and $D$ is unique.  
	\end{proof}
	
	\begin{cor}\label{corcd}
		Let $(S,\La)$ be a spectral pair with $pq\mid\mid |S|=|\La|>pq$. Then $S_p=qD$, where $D\ne \Z_p^2$ is a nonempty set ($c=0$).
	\end{cor}
	\begin{proof}
		$D \not\equiv 0$ since $p^2 \nmid |S|$. Further $c=0$, otherwise we have a $\Z_q^2$ coset containing more than $q$ elements of $S$, which implies $q^2 \mid |S|$ since $S$ is spectral, a contradiction. 
		As a corollary of Lemmata \ref{lemp2q2} and \ref{lemleaf} we obtain that the intersection of $S$ with each $\Z_q^2$ coset is either of size $0$ or $q$. This follows again from the fact that if  a $\Z_q^2$ coset contains more than $q$ elements, then $q^2 \mid |S|$ by Lemma \ref{lemp2q2}. 	
	\end{proof}
	Note that the same holds for $\La$ since we only used the condition that $S$ is a spectral. 
	Further the role of $p$ and $q$ can also be switched in Corollary \ref{corcd}.
	
	Assume both $S$ and $\Lambda$ have the previously described structure. 
	The $\Z_q^2$ cosets can be indexed by the elements of $\Z_p^2$ so we write $K_{a}=((\Z_q^2+a) \cap S)-a$ for $a \in \Z_p^2$.  
	
	\subsubsection{Special leaf structure}
	In this subsection we handle the following case.
	
	{\bf Assumption A}:
	There exists an $0 \ne u \in \Z_p^2$ such that for every $b \in \Z_p^2$ and $\lambda, \lambda' \in \Z_p$ we have either $K_{b+\lambda u}=\emptyset$ or $K_{b+\lambda' u}=\emptyset$ or $K_{b+\lambda u}=K_{b+\lambda' u}$.
	
	Note that the role of $S$ and $\La$ can be exchanged in this section so the statements for $S$ remain valid for $\La$. Hence, we prove our statements only for $S$. 
	Let $0\ne u'\in \Z_p^2$ be orthogonal to $u$. 
	\begin{lem}\label{lemunotequid}
		$S$ projected to $u'$ is not equidistributed.  
	\end{lem}
	\begin{proof}
		Assume indirectly it is equidistributed. Then each  $\langle u,\Z_q^2\rangle$ coset contains $kq$ elements of $S$. Since the $\Z_q^2$ cosets in $\langle u,\Z_q^2\rangle$ intersect $S$ the same way by Assumption A, we obtain $k \mid p$ since the weight of every point of the multiset $S_q$ is $p$ or $0$. This contradicts the fact that $pq<|S|<p^2q$.
	\end{proof}
	
	\begin{lem}\label{lemns}
		Let $S$ be a set with $pq\mid\mid |S|$ and $pq<|S|<pq\min\{p,q\}$. Suppose that Assumption A holds. Then $S$ is not spectral. 
	\end{lem}
	\begin{proof}
		Let $v \in \Z_q^2$. Let $S_{\langle u,v  \rangle}$ 
		denote the projection of $S$ to $\langle u,v  \rangle$.
		By Assumption A, $S_{\langle u,v  \rangle}(i,j)=l \cdot h_i(j)$, where $h_i: \Z_q \to \N$  and $l$ is the number of nonempty $\Z_p^2$ cosets in the $\Z_p^2\times \Z_q$ coset projecting to these points. Further $\sum_{x \in  \Z_q }h_i(x)=q$, since by Corollary \ref{corcd}, $S$ intersects with each $\Z_p^2$ coset in either 0 or $q$ times. 
		This implies that if there is an $i \in \Z_p$ such that $S_{\langle u,v  \rangle}(i,j) > 0$ for every $j \in \Z_q$, then $S_{\langle u,v  \rangle}(i,j)$ is constant as a function of $j$, and if $S_{\langle u,v  \rangle}(x)>0$ for every $x \in \Z_p \times \Z_q$, then $S_{\langle u,v  \rangle}(i,j)$ is constant for every $i$ as a function of $j$. This implies that $\chi_v(S)=0$. 
		
		Assume $\chi_v(S) \ne 0$. Since $|S|=|\La|>pq$, by Lemma \ref{lem:obs} either $v\in\D(\La)$ or $u\in \D(\La)$ or $u+v\in \D(\La)$. The first two are excluded as $\chi_u(S) \ne 0$ and $\chi_v(S)\ne 0$. Thus $u+v\in \D(\La)$ holds, which implies that $\chi_{u+v}(S)=0$.
		It follows from the previous argument that $S_{\langle u,v  \rangle}(i,j)=0$ for at least one pair $(i,j)$, otherwise $\chi_v(S)=0$. By Lemma \ref{lemcuberule1}, $S_{\langle u,v  \rangle}$ can be written as the sum of characteristic functions of $\Z_p$ cosets and $\Z_q$ cosets.
		
		The $\Z_p$ and $\Z_q$ cosets containing $(i,j)\in \Z_p \times \Z_q$ are not in the sum of cosets determining 
		$S_{\langle u,v  \rangle}$. Since $\chi_u(S) \ne 0$ and $\chi_{v}(S) \ne 0$, this sum contains a $\Z_q$ coset and a $\Z_p$ coset as well. Since $S_{\langle u,v  \rangle}$ is positive in this $\Z_q$ coset, as we have seen it implies that $S_{\langle u,v  \rangle}$ is constant on this $\Z_q$ coset, which contradicts the fact that it is the sum of $\Z_p$ and $\Z_q$ cosets, where some $\Z_p$ cosets are missing, while some are present (with multiplicity at least 1). This contradiction shows that under Assumption A and $pq\mid\mid|S|>pq$, for every $0\ne v\in \Z_q^2$, $\chi_{v}(S)= 0$.  
		By Lemma \ref{lemp2mid}, this implies that $q^2 \mid |S|$, which again contradicts the assumption $pq\mid\mid|S|$.
		
		Thus we obtain that $S$ is not  spectral. 
	\end{proof}
	\subsubsection{General case}
	\begin{prp}\label{prplotsofzeros}
		Let $(S,\La)$ be a spectral pair with $pq\mid\mid |S|=|\La|>pq$. 
		Suppose $\chi_{u+v}(S) \ne 0$, where $0 \ne u \in \Z_p^2$ and $0 \ne u \in \Z_q^2$. Then $\chi_u(S)=0$ and $\chi_v(S)=0$.
	\end{prp}
	\begin{proof} 	
		Assume $\chi_u(S) \ne 0$. Combining this with $\chi_{u+v}(S) \ne 0$ we obtain that if $K_a$ and $K_b$ are contained in the same $\langle u, \Z_q^2 \rangle$ coset (where $K_x=((\Z_q^2+a)\cap \La)-x$), then for every $w \in \Z_q^2$, at least one of $K_a \cap ( w+ \langle v \rangle)$ and $K_b \cap ( w+ \langle v \rangle)$ is empty.
		
		Thus if $\la \in \La$, then $(\la+\langle u,v\rangle) \cap \La \subseteq \la +\langle v \rangle$. By Corollary \ref{corcd}  $\La$ has $p$ elements in $\la+ \Z_p^2$ so it has exactly one element in $(\la+ \Z_p^2)\cap x+\langle u, \Z_q^2 \rangle$ for every $x \in \Z_p^2 \times \Z_q^2$. Further every element of $\La$ contained in $x +\langle u, v \rangle$ is contained in a single $\langle v \rangle$-coset. Thus Assumption A holds for $\La$ with $v$ playing the role of $u$ and $q$ the one of $p$. Such a set cannot be spectral by Lemma \ref{lemns}. 
		
		Therefore, if $(S, \La)$ is a spectral pair and $\chi_{u+v}(S)\ne 0$, where $0 \ne u \in \Z_p^2$ and $0 \ne u \in \Z_q^2$, then $\chi_u(S)=0$. Changing the role of $u$ and $v$ implies that $\chi_v(S)=0$, as well, finishing the proof of Proposition \ref{prplotsofzeros}.
	\end{proof}
	
	\begin{prp}\label{propfin}
		If $pq\mid\mid|S|$ so that $pq<|S|<pq\min\{p,q\}$, then $S$ is not spectral. 
	\end{prp}
	\begin{proof} 
		Suppose that $S$ is spectral.
		Without loss of generality we can assume that $p<q$. As $p^2\nmid |S|$, there is a $0 \ne u \in \Z_p^2$ such that $\chi_u(S)\ne 0$ by Lemma \ref{lemp2mid}.
		If $S$ is spectral and $pq\mid\mid|S|>pq$, then   $\chi_{u+v}(S)= 0$ for every and $0 \ne u \in \Z_q^2$, by Proposition \ref{prplotsofzeros}. Hence the projection $T$ of $S$ to $\langle u, \Z_q^2 \rangle$ the conditions of Proposition \ref{prop1} hold. Thus we obtain that for $g_i(z)=T(i+z)$ ($z\in \Z_q^2$), $g_i-g_j$ is a constant function for every $i,j \in \Z_p$. If $g_i=g_j$ for every $i,j \in \Z_p$, then $\chi_u(S)=0$, a contradiction. Otherwise there are some  $i,j \in \Z_p$ such that $g_i\ge g_j+1\ge 1$, since $g_j\ge 0$. Thus there is a  coset of $\Z_q^2$ in $\Z_p\times \Z_q^2$ that contains at least $q^2$ element of $T$ (i.e., there is a coset of $\Z_p\times \Z_q^2$  in $\Z_p^2\times \Z_q^2$   containing at least $q^2$ element of $S$). On the other hand, every coset of cardinality $pq^2$ contains at most $pq$ elements of $S$, otherwise there is a coset of $\Z_q^2$ containing $q+1$ elements of $S$ which would imply that every direction appears in $S$ and, by Corollary \ref{corequigen}, we get $q^2\mid |S|=|\La|$. This is a contradiction, as for $p<q$ we have $pq<q^2$. This contradiction shows that $S$ is not spectral.  
	\end{proof}
	Proposition \ref{propfin} finishes the proof of Case 5 and the spectral tile direction of Theorem \ref{t1}.
	
	\section{Tile-spectral direction}\label{sec5}
	Let $A \bigoplus B=G$ be a tiling of $G$. Clearly, every translates of $A$ is a tiling partner of $B$ and since the role of $A$ and $B$ is symmetric we may assume $0 \in A \cap B$.
	Plainly,  $\widehat{ 1}_A\cdot\widehat{ 1}_B=|G|\cdot\delta_{1}$, where $\delta$ denotes the Dirac delta and $1$ denotes the trivial representation of $G$. 
	We recall some elementary result from \cite{KS2021}.
	\begin{lem}[Lemma 3.2 of \cite{KS2021}]\label{lemsubgroup}
		Let $G$ be an abelian group, which is the direct sum of subgroups $B$ and $P$. Let $A \subset G$ such that $A\bigoplus B=G$. Then $A$ is spectral and its spectrum is $P$.
	\end{lem}
	Note that $\Z_p^2 \times \Z_q^2$ is a complemented group, which means that if $K$ is a subgroup of $G$, then there is a subgroup $L$ of $G$ with $K\bigoplus L=G$. 
	Thus Lemma \ref{lemsubgroup} implies the following. 
	\begin{cor}\label{corspect}
		If $A$ is a subset of $G=\Z_p^2 \times \Z_q^2$ such that $A\bigoplus B=G$, where $B$ is a subgroup of $G$, then $A$ is spectral. 
	\end{cor}
	We are considering finite groups,  so a set \(A\) is spectral with spectrum \(\Lambda\) when \(|A| = |\Lambda|\), and \( \hat{1}_A(\chi) = 0\) for every nonzero \(\chi \in \Lambda - \Lambda\).
	In each case we find a subgroup of \(\Lambda \leq \widehat{G}\) where \(|\Lambda| = |A|\) and \(\hat{1}_A\) vanishes on all nonzero elements of \(\Lambda\), making \(\Lambda\) a spectrum for \(A\).
	
	One can see that $G$ and $\{1\}$ are spectral sets and tiles so in these cases the tile-spectral direction automatically holds.
	Therefore, we may assume $A \subsetneq G$ and $B \subsetneq G$. Hence, there are 
	$1\ne \chi,\chi' \in \widehat{G}$ with $\chi(A)\ne 0$ and $\chi'(B)\ne 0$ by Lemma \ref{lemeqdsubgrp}. Then it follows from $\widehat{{\bf 1}}_A\cdot\widehat{{\bf 1}}_B=|G|\cdot\delta_{1}$ that $\chi(B) =0$ and $\chi'(A)=0$. 
	
	Now we prove the tile-spectral direction by distinguishing cases according to $|A|$.
	For the first two cases we consider \(G = H \times K\) where \(H\) is the Sylow \(p\)-subgroup of \(G\) with order \(p^n\).
	
	\medskip
	\noindent
	\underline{{\bf Case 1:} \(|A| = p\)} \\
	\(A\) has a tiling partner \(B\) of size \(\frac{|G|}{p}\).
	In particular, since \(p^n \nmid |B|\), there must be some character \(\chi_a\) of order \(p^k\) with \(0 < k \le n\) where \(\chi_a(B) \neq 0\), so \(\chi_a(A) = 0\). Then
	\(A\) must project to \(G / \chi_a^\perp\) as a sum of \(\Z_p\) cosets, and as \(|A| = p\), this is a single coset.
	Where \(\chi_b\) is a character of order \(p\) in \(\langle \chi_a \rangle\),
	\(A\) must then be equidistributed among cosets of \(\chi_{b}^\perp\).  Therefore \(\chi_{b}(A) = 0\) and \(\langle \chi_{a} \rangle\) is a spectrum for \(A\).
	
	In our case $k=1$ since the Sylow $p$-subgroup of $\Z_p^2 \times \Z_q^2$ is of exponent $p$. Notice this argument shows that every subset of prime cardinality of every finite abelian group is spectral if it is a tile.
	
	\medskip
	\noindent
	\underline{{\bf Case 2:} \(|A|=p^n\)} \\
	
	In this case, \(A\) has a tiling partner \(B\) where \(|B| = |K|\) and \(p \nmid |B|\).
	\(\hat{1}_B\) cannot vanish on any character of order \(p^k\) for \(1 \leq k\), so \(\hat{1}_A\) vanishes on all such characters.  This means that \(\hat{1}_A\) vanishes on all non-zero characters in \(\hat{H}\), and since \(|\hat{H}| = |A|\), this makes \(\hat{H}\) a spectrum for \(A\).
	
	When \(G = \Z_p^2 \times \Z_q^2\), this covers the case where \(|A|\) is either \(p^2\) or \(q^2\).
	
	\medskip
	\noindent
	Note that in the remaining cases, we take $G=\Z_p^2 \times \Z_q^2$.
	
	\medskip
	\noindent
	\underline{{\bf Case 3:} $|A|=|B|=pq$} \\
	
	We prove that $A$ is spectral. For $B$ the proof is analogue. 
	
	Since $p^2 \nmid |B|$ and $q^2 \nmid |B|$, there exist $u\in \Z_p^2\setminus\{0\}$ and $v\in \Z_q^2\setminus\{0\}$ such that $\chi_u(B)\ne 0$ and $\chi_v(B)\ne 0$. As $\phi(A)\phi(B)=0$ for every character $\phi \ne 1$ we have $\chi_u(A)= 0$ and $\chi_v(A)=0$. Now, it would be enough to prove that $\chi_{u+v}(A)=0$ since then the subgroup of order $pq$ generated by $u$ and $v$ would be a spectrum for $A$ by Lemma \ref{lemcharacterequiv}.
	
	For contradiction, we assume that  $\chi_{u+v}(A)\ne 0$, hence $\chi_{u+v}(B)= 0$.
	This implies that \(B\) projects to \(G / \chi_{u+v}^\perp\) as sum of $\Z_p$ and $\Z_q$ cosets by Lemma \ref{lemcuberule1}.
	As $|B|=pq$, it simply follows that this is either the sum of \(q\) many $\Z_p$ cosets or the sum of \(p\) many $\Z_q$ cosets. The former case implies that $\chi_u(B)=0$, the latter case implies that $\chi_v(B)=0$. Both contradicts our assumption. 
	
	We conclude $\chi_{u+v}(B) \neq 0$ and \(\chi_{u+v}(A) = 0\), and hence \(\langle \chi_{u+v}\rangle\) is a spectrum for \(A\).

	\medskip
	\noindent
	\underline{{\bf Case 4:} $|A|=pq^2$ (or $|A|=p^2q$)} \\
	
	Such a tiling set has tiling partner \(B\) where \(|B| = p\).
	Since \((A,B)\) is a tiling pair and \(q \nmid |B|\), \(\hat{1}_A\) vanishes on every character of order \(q\).
	Since \(p^2 \nmid |B|\), there is a character \(\chi\) of order \(p\) where \(\hat{1}_B(\chi) \neq 0\) and \(\hat{1}_A(\chi) = 0\).
	First consider the case where \(\hat{1}_A\) also vanishes on every character of order \(pq\).
	Then \(\hat{1}_A\) vanishes on all nonzero elements of \(\langle \chi \rangle \oplus \Z_q^2\), which is a subgroup of \(\widehat{G}\) with same size as \(A\) and hence is a spectrum for \(A\).
	
	If there is some character \(\alpha\) of order \(pq\) where \(\hat{1}_A\) does not vanish, then \(\hat{1}_B(\alpha) = 0\).
	B then projects to \(G / \alpha^\perp\) as a sum of \(\Z_p\) and \(\Z_q\) cosets by Lemma \ref{lemcuberule1}, and since \(|B| = p\), this will be a \(\Z_p\) coset.
	\(B\) is therefore contained in a coset of \(\left(\alpha^p\right)^\perp \cong \Z_p^2 \times \Z_q\).
	Choose some \(\xi \in G \backslash \left(\alpha^p\right)^\perp\), then \(G = \bigsqcup_{i \in \Z_q} \left(\alpha^p\right)^\perp + i\xi\).
	For \(i \in \Z_q\), define \(A_i \coloneqq A \cap \left(\left(\alpha^p\right)^\perp + i\xi \right) - i\xi\).
	Then for each \(i\), \((A_i, B)\) is a tiling pair for \(\left(\alpha^p\right)^\perp\).
	Note that \(\hat{1}_A\) must vanish on any point where \(\hat{1}_{A_i}\) does for all \(i \in \Z_q\). 
	
	Since \(B\) tiles \(\left(\alpha^p\right)^\perp\), if there is some character \(\beta\) of order \(pq\) in  \(\widehat{\left(\alpha^p\right)^\perp}\) where \(\hat{1}_B\) vanishes, then \(B\) projects to \(\left(\alpha^p\right)^\perp / \beta^\perp \cong \Z_{pq}\) as a union of \(\Z_p\) and \(\Z_q\) cosets by Lemma \ref{lemcuberule1}, and again due to the size of \(B\), this will be a single \(\Z_p\) coset.
	This means that \(B\) is contained in \(\Z_p^2\) and hence determines a direction of order \(p\).
	This direction is missed by \(A\) 
	by Theorem 3.1 from \cite{atenetal}. Note that the equivalence of case (h) in Theorem 3.1 is only proved in the case $n=p$ but the proof works verbatim for an arbitrary natural number using Lemma \ref{lemcharacterequiv}. 
	Since \(|A| = \frac{|G|}{p}\), \(A\) is a graph and hence is spectral. 
	
	Otherwise each \(\hat{1}_{A_i}\) vanishes on every character of order \(pq\) in \(\widehat{\left(\alpha^p\right)^\perp}\) and hence so does \(\hat{1}_A\).
	We claim that $\hat{1}_A$ vanishes on every nontrivial elements of $\langle \chi \rangle \oplus \Z_q^2$. We have seen above this for characters of order $q$ and $pq$. The characters of order $p$ of this subgroup are powers of $\chi$, so the result follows from Lemma \ref{lemcharacterequiv}.
	This in turn implies \(\langle \chi \rangle \oplus \Z_q^2\) is a spectrum for \(A\) since $|\langle \chi \rangle \oplus \Z_q^2|=|A|$.
	
	\section*{Acknowledgement}
	\ \  \ \ The second author was supported by Premium Postdoctoral Fellowship of the Hungarian Academy of Sciences and by the Hungarian National
	Foundation for Scientific Research, Grant No. K124749.

	The third author was supported by the J\'anos Bolyai Research Fellowship and by the New National
	Excellence Program under the grant number \'UNKP-20-5-ELTE-231. 
	
	Application Domain Specific Highly Reliable IT Solutions” project  has  been implemented with the support provided from the National Research, Development and Innovation Fund of Hungary, financed under the Thematic Excellence Programme TKP2020-NKA-06 (National Challenges Subprogramme) funding scheme.

\end{document}